\documentclass[10pt]{amsart}
\usepackage[cp1250]{inputenc}
\usepackage{latexsym}
\usepackage{amsmath}
\usepackage{amsfonts}
\usepackage{amscd}
\usepackage{amsthm}
\usepackage{amssymb}
\usepackage{enumerate}
\usepackage{mathabx}
\usepackage{amsrefs}
\usepackage{color}
\usepackage{soul}

\usepackage[T1]{fontenc}
\author{\.Z.~Fechner, L.~Sz\'ekelyhidi}
\title[Gajda-Type Equation on Topological Groups]{A Generalization of Gajda's 
Equation on Commutative Topological Groups}
\dedicatory{Dedicated to the memory of Professor Pl. Kannappan}
\address{Institute of Mathematical Finance, Ulm University, Helmholtzstraße 18, 89081 Ulm, Germany}
\email{zywilla.fechner@uni-ulm.de}
\address{Institute of Mathematics, University of Debrecen, Egyetem t\'er 1, 4032 Debrecen, Hungary --- Department of Mathematics, University of Botswana, 4775 Notwane Rd. Gaborone, Botswana }
\email{lszekelyhidi@gmail.com}

\theoremstyle{plain}
\newtheorem{theorem}{Theorem}

\theoremstyle{remark}

\newcommand{\C}{\mathbb{C}}

\renewcommand{\(}{\left(} \renewcommand{\)}{\right)}
 
\begin{document}

\begin{abstract}
In the present paper we deal with the following generalization of the sine-cosine equation
\begin{equation*}
\int f_1(x+y-t)+f_2(x-y+t) d\mu(t)=g(x)h(y)
\end{equation*}
for complex valued functions $f_1$, $f_2$, $g$ and $h$ defined on a 
commutative topological group $G$, where $\mu$ is a complex measure defined on $G$. 
\end{abstract}

\maketitle

\section{Introduction}
Let $G$ be an arbitrary group. One of the most famous trigonometric functional equations is 
\emph{d'Alembert's functional equation}:
\begin{equation}\label{eq:DalembertKlasyczny}
f(x+y)+f(x-y)=2f(x)f(y),\quad x,y \in G.
\end{equation}
Equation \eqref{eq:DalembertKlasyczny}, also called the \emph{cosine equation}, 
as $f=\cos$ satisfies \eqref{eq:DalembertKlasyczny} in the real-to-real case, has 
been investigated by many authors. Pl. Kannappan \cite[Kannappan]{MR0219936} considered d'Alembert functional 
equation if the unknown function is defined on an arbitrary commutative group and takes values in the field of complex numbers under certain commutative-type condition. 
\vskip.3cm

One of the possible generalizations of d'Alembert's functional equation is  {\it Wilson's functional equation}
\begin{equation}\label{eq:WilsonKlasyczny}
g(x+y)+g(x-y)=2g(x)f(y),\quad x,y \in G. 
\end{equation}
This is called also the \emph{sine-cosine functional equation} as $g=\sin$ and $f=\cos$ is a solution in the real-to-complex case. 
It is worth underlining that the main difficulty in solving Wilson's-type equations is to give a description of the \hbox{function $g$.}
This is not obvious even in the real-to-real case. One possible method is to use spectral synthesis. This was discussed in details in \cite[Sz\'ekelyhidi]{MR1113488}.
For further discussion of generalization of cosine and sine equations for unknown mappings defined on non-commutative groups see \cite[Stetk\ae r]{Stet13} and references therein. 
\vskip.3cm

Observe that \eqref{eq:DalembertKlasyczny} can be written as the convolution of the unknown function $f$ with a measure:
$$ 	f*\(\frac{1}{2}\,\delta_y\)(x)+f*\(\frac{1}{2}\,\delta_{-y}\)(x)=f(x)f(y),\quad x,y \in G\,, 
$$
where $\delta_y$ denotes the Dirac measure concentrated at $y$. Our aim is to generalize this equation by substituting the Dirac measure by a -- more or less -- arbitrary measure.
\vskip.3cm

In the same manner as for the d'Alembert equation we can rewrite \st{this} equation \textcolor{red}{\ref{eq:WilsonKlasyczny}} as convolution of the unknown function with the Dirac measure, however, this time we have two unknown functions, namely
$$ 	g*\(\frac{1}{2}\,\delta_y\)(x)+g*\(\frac{1}{2}\,\delta_{-y}\)(x)=g(x)f(y),\quad x,y \in G\,, 
$$
Hence our generalization works in two directions: we have more unknown functions and an "almost" arbitrary measure. 
\vskip.3cm

Motivation for this investigation is the following equation:
\begin{equation}
	(f*\mu_y)(x)+(f*\widecheck{\mu}_y)(x)=f(x)f(y),\quad x,y \in G\,,
	\label{eq:Gajda}
\end{equation}
which was introduced and solved by Z. Gajda in \cite[Gajda]{MR1065469} for essentially bounded measurable functions defined on a locally compact abelian group. 
Here $\mu_y$, resp. $\widecheck{\mu}$ denotes the {\it translate}, resp. the {\it inversion} of the measure $\mu$. The main tool used by Gajda was the Wiener Tauberian theorem, and he expressed 
the solution  as a linear combination of characters of the group with coefficients depending on the \hbox{measure $\mu$.}
\vskip.3cm

The next attempt was the investigation of the Gajda--type generalization of Wilson's functional equation, namely
\begin{equation}
	(g*\mu_y)(x)+(g*\widecheck{\mu}_y)(x)=g(x)f(y),\quad x,y \in G\,,
	\label{eq:WilsonSplot}
\end{equation}
which has been discussed in \cite[Fechner]{MR2515239}. In \cite[Fechner]{MR2807038} the following equation
\begin{equation}
	(f*\mu_y)(x)+(f*\widecheck{\mu}_y)(x)=g(x)f(y),\quad x,y \in G\,,
	\label{eq:WilsonSplotChange}
\end{equation}
has been investigated as a counterpart of \eqref{eq:WilsonSplot}.
\vskip.3cm

In this paper we shall consider the integral-functional equation
\begin{equation}\label{sincos1}
\int [f_1(x+y-t)+f_2(x-y+t) ] d\mu(t)=g(x)h(y),\quad x,y \in G\, ,
\end{equation} where $f_1,f_2,g,h\colon G\to \C$ are unknown functions 
and $\mu$ is a complex measure on $G$, 
or equivalently, we use the convolution form
\begin{equation}\label{sincos2}
(f_1*\mu)(x+y)+(\widecheck{f}_2*\mu)(x-y)=g(x)h(y),\quad x,y \in G\,,
\end{equation}
where $\widecheck{f}(x)=f(-x)$ for every $x$ in $G$, and we have interchanged the roles of $g$ and $h$.  This equation is a common generalization of \eqref{eq:Gajda}, \eqref{eq:WilsonSplot} and \eqref{eq:WilsonSplotChange}.
\vskip.3cm

In the forthcoming paragraphs we shall use the results in \cite[Sz\'ekelyhidi]{MR1113488} to give a complete 
description of the solutions of \eqref{sincos2}. The idea is that, by introducing 
the functions $F_1=f_1*\mu$ and $F_2=\widecheck{f}_2*\mu$, we have the 
functional equation 
\begin{equation}\label{dAlem1}
F_1(x+y)+F_2(x-y)=g(x)h(y),\quad x,y \in G\,,
\end{equation}
where $F_1,F_2$ have similar regularity properties like $f_1$ and $f_2$. Having 
the general solution of equation \eqref{dAlem1} we have to solve the 
inhomogeneous convolution equations, which define $F_1$ and $F_2$. 
\vskip.3cm

We may impose different conditions on the topology of $G$, on the functions and on the measure so that the 
integrals exist. If $G$ is locally compact, then we suppose that $\mu$ is a 
compactly supported Borel measure and the unknown functions are continuous. 
In particular, if $G$ is a discrete group, then $\mu$ is finitely supported and no 
conditions on the unknown functions are assumed. If $G$ is an arbitrary 
topological group, then $\mu$ is a Borel measure and the unknown functions 
are $\mu$-integrable.

\section*{Notation and terminology}
For a given function $f\colon G\to\C$, as above, we use the notation
$$
\widecheck{f}(x)=f(-x),\quad x\in G 
$$
and
\begin{equation*}\label{evenodd}
f_e(x)=\frac{1}{2}\(f(x)+\widecheck{f}(x)\),\hskip.2cm f_o(x)=\frac{1}{2}\(f(x)-\widecheck{f}(x)\)
\end{equation*}
for each $x$ in $G$, and we call these functions the {\it even part}, and the {\it 
odd part} of $f$, respectively. We have, obviously, $f=f_e+f_o$. 
\vskip.3cm

Let $G$ be a topological group. We call a nonzero function $m:G\to\C$ an {\it exponential}, if it satisfies
\begin{equation*}
m(x+y)=m(x)m(y)
\end{equation*}
for each $x,y$ in $G$. It is easy to see that an exponential never vanishes. A 
 function $a:G\to \C$ is called {\it additive}, if it satisfies
\begin{equation*}
a(x+y)=a(x)+a(y)
\end{equation*}
for each $x,y$ in $G$. For more about exponentials and additive functions see 
\cite[Sz\'ekelyhidi]{MR1113488}. In particular, we shall use the result, which says that
a representation of a function in the form $x\mapsto \bigl(a(x)+b\bigr) m(x)$ 
is unique, whenever $m$ is an exponential, $a$ is additive and $b$ is a complex 
number (see \cite[Sz\'ekelyhidi]{MR1113488}, Lemma 4.3, p.~41).  It follows that  functions of this form are linearly dependent if and only if they have the same exponential, and the corresponding $a(x)+b$ factors are linearly dependent.
\vskip.3cm

We shall deal with functions $T:G\to\C$, which are constant on the cosets of the subgroup $2 G$. Such functions we will call {\it $2 G$-periodic}.  Obviously, $2 G$-periodic functions are even. In particular, an exponential is $2 G$-periodic if and only if it is even. If $G$ is {\it 2-divisible}, that is $G=2 G$, then $2 G$-periodic functions are constant, $2 G$-periodic additive functions are identically zero, and $2 G$-periodic exponentials are identically $1$.
\vskip.3cm

If $\mu$ is a measure on $G$ with the property that the exponential $m$ is  integrable with respect to $\mu$, then we use the notation 
\begin{equation*}
\widehat{\mu}(m)=\int \widecheck{m} d\mu\,.
\end{equation*}
This is the standard notation used for the {\it Fourier--Stieltjes transform}, which is the restriction of $\widehat{\mu}$ to the dual of $G$ (see e.g. \cite[Rudin]{MR1038803}). We note that convolution is defined in the usual manner
\begin{equation*}\label{convo}
(f*\mu)(x)=\int  f(x-t)\,d\mu(t)\,,
\end{equation*}
whenever it exists.


\section{Solution of equation \eqref{dAlem1}}

In this section we describe the solutions of the functional equation 
\eqref{dAlem1} using the results of \cite[Sz\'ekelyhidi, Section 11]{MR1113488}.  

\begin{theorem}\label{dAlemsol}
Let $G$ be an abelian group and let $F_1,F_2,g,h\colon G\to\C$ be functions satisfying the functional equation \eqref{dAlem1} for each $x,y$ in $G$. Then the functions \hbox{$F=F_1+F_2$} and $H=F_1-F_2$ satisfy the functional equations
\begin{equation}\label{dAlemeven}
F(x+y)+F(x-y)=2g(x)h_e(y)\,,
\end{equation}
and
\begin{equation}\label{dAlemodd}
H(x+y)-H(x-y)=2g(x)h_o(y)
\end{equation}
for each $x,y$ in $G$.
\end{theorem}

\begin{proof}
Substituting $y$ by $-y$ in \eqref{dAlem1}, and then adding, resp. subtracting the new equation to, resp. from \eqref{dAlem1} we obtain \eqref{dAlemeven}, resp. \eqref{dAlemodd}. 
\end{proof}

First we describe the solutions of \eqref{dAlemeven}.

\begin{theorem}\label{dAlemevensol}
Let $G$ be an abelian group and let $F,g,h\colon G\to \C$ be functions satisfying the 
functional equation
\begin{equation}\label{dAlemevene}
F(x+y)+F(x-y)=2g(x)h_e(y)
\end{equation}
for each $x,y$ in $G$. Then we have the following possibilities:
\begin{enumerate}[i)]
\item
\begin{eqnarray*}
F(x)&=&\gamma(\alpha m(x)+\beta m(-x))\\
g(x)&=&\alpha m(x)+\beta m(-x)\\
h_e(x)&=&\frac{\gamma}{2}(m(x)+m(-x))\,,
\end{eqnarray*}
\item
\begin{eqnarray*}
F(x)&=&(a(x)+\alpha \beta) m_0(x)\\
g(x)&=&\left[\frac{1}{\alpha} a(x)+\beta\right] m_0(x)\\
h_e(x)&=&\alpha m_0(x),\quad \alpha \neq 0
\end{eqnarray*}
\item
\begin{eqnarray*}
F(x)&=&0\\
g(x)&=&0\\
h&=&\text{arbitrary function}
\end{eqnarray*}
\item
\begin{eqnarray*}
F(x)&=&0\\
g&=&\text{arbitrary function}\\
h&=&\text{arbitrary odd function}
\end{eqnarray*}
\end{enumerate}
for each $x$ in $G$, where $\alpha, \beta, \gamma$ are complex numbers, 
$\gamma\ne 0$, $a\colon G\to\C$ is an additive function and $m,m_0\colon G\to \C$ are exponentials with 
$m\ne \widecheck{m}$ and $m_0=\widecheck{m}_0$. 
Conversely, any functions with the given properties satisfy the functional 
equation \eqref{dAlemevene}. If, in addition, $G$ is a topological 
group, $F\ne 0$, and $g\ne 0$ is continuous, then $F,a,m$ are continuous, too. 
If $G$ is a locally compact abelian group, $F\ne 0$ and $g\ne 0$ is 
measurable, then $a,m,F,g,h_e$  are continuous.
\end{theorem}

\begin{proof}
The last two cases are obvious, so we suppose that $F\ne 0$. By Theorem 11.1, p. 97
in \cite[Sz\'ekelyhidi]{MR1113488}, it follows that $F$ has one of the following forms:
\begin{enumerate}[i)]
\item $F(x)=\alpha m(x)+\beta m(-x)$\,,
\item $F(x)=(a(x)+b)m_0(x)$
\end{enumerate}
for each $x$ in $G$, where $\alpha, \beta,b$ are complex numbers, 
$a\colon G\to\C$ is an additive function, \hbox{$m,m_0\colon G\to\C$} are exponentials, further $m\ne 
\widecheck{m}$ and $m_0=\widecheck{m}_0$. As $F$ is nonzero, hence 
$h_e(0)\ne 0$, which implies that $g$ has the same form with some different 
constants, and as $g$ is nonzero, hence the same holds for $h_e$. Substitution 
of the given expressions for $F,g,h_e$ into \eqref{dAlemevene} and renaming 
the constants we obtain our statement.
\vskip.3cm

The regularity statements follow immediately from Lemma 5.5  and Theorem  5.10 in  \cite[Sz\'ekelyhidi]{MR1113488}.
\end{proof}

Now we describe the solutions of \eqref{dAlemodd}.

\begin{theorem}\label{dAlemoddsol}
Let $G$ be an abelian group and let $H,g,h\colon G\to \C$ be functions satisfying the 
functional equation
\begin{equation}\label{dAlemodde}
H(x+y)-H(x-y)=2g(x)h_o(y)
\end{equation}
for each $x,y$ in $G$. Then we have the following 
possibilities:
\begin{enumerate}[i)]
\item
\begin{eqnarray*}
H(x)&=&\alpha \gamma m(x)-\beta \gamma m(-x)+T(x)\\
g(x)&=&\alpha m(x)+\beta m(-x)\\
h_o(x)&=&\frac{\gamma}{2}(m(x)-m(-x))\,,
\end{eqnarray*}
\item
\begin{eqnarray*}
H(x)&=&(a(x)+b)m_0(x)+T(x)\\
g(x)&=&\frac{1}{\alpha} m_0(x)\\
h_o(x)&=&\alpha a(x)m_0(x)\,,
\end{eqnarray*}
\item
\begin{eqnarray*}
H(x)&=&T(x)\\
g(x)&=&0\\
h&=&\text{arbitrary function}\,,
\end{eqnarray*}
\item
\begin{eqnarray*}
H(x)&=&T(x)\\
g&=&\text{arbitrary function}\\
h&=&\text{arbitrary even function}
\end{eqnarray*}
\item
\begin{eqnarray*}
H(x)&=&0\\
g(x)&=&0\\
h&=&\text{arbitrary function}\,,
\end{eqnarray*}
\item
\begin{eqnarray*}
H(x)&=&0\\
g&=&\text{arbitrary function}\\
h&=&\text{arbitrary even function}
\end{eqnarray*}
\end{enumerate}
for each $x$ in $G$, where $\alpha, \beta, \gamma$ are complex numbers, 
$\alpha\ne 0$,
$a\colon G\to \C$ is a nonzero additive function, $m,m_0\colon G\to \C$ are exponentials, 
$m\ne \widecheck{m}$, $m_0= \widecheck{m}_0$, further
$T\colon G\to\C$ is a $2 G$-periodic function. Conversely, the functions given with these properties satisfy 
the functional equation \eqref{dAlemevene}. If, in addition, $G$ is a  topological group and $g,h_o\ne 0$ are continuous, then $a,m,m_0$ are 
continuous,  too. If $G$ is a locally compact abelian group and $g,h_o\ne 0$ are 
measurable,  then $a,m,m_0,g,h_o$ are continuous.
\end{theorem}

\begin{proof}
Similarly, like in the proof of the previous theorem, the last four cases are 
obvious, so we suppose that $H, g, h_o\ne 0$. Then, by Theorem 11.2 in \cite[Sz\'ekelyhidi]{MR1113488}, 
we have that $H$ has one of the following forms:
\begin{enumerate}[i)]
\item $H(x)=\alpha m(x)+\beta m(-x)+T(x)$\,,
\item $H(x)=(a(x)+b)m_0(x)+T(x)$
\end{enumerate}
for each $x$ in $G$, where $\alpha, \beta,b$ are complex numbers, 
$a\colon G\to\C$ is an additive function, \hbox{$m, m_0\colon G\to\C$} are exponentials, further $m\ne 
\widecheck{m}$ and $m_0=\widecheck{m}_0$, and finally $T\colon G\to\C$ is a 
$2 G$-periodic function. As $g$ and 
$h_o$ are nonzero, hence they have the same form with some different 
constants. Substitution of the given expressions for $H,g,h_o$ into 
\eqref{dAlemevene} and renaming the constants yields the statement.
\vskip.3cm

The regularity statements follow immediately from Lemma 5.5  and Theorem 
5.10  in  \cite[Sz\'ekelyhidi]{MR1113488}.
\end{proof}

Now we are in the position to describe all solutions of the functional \hbox{equation \eqref{dAlem1}.}

\begin{theorem}\label{dAlemsol1}
Let $G$ be an abelian group and let $F_1,F_2,g,h\colon G\to\C$ be functions 
satisfying the functional equation \eqref{dAlem1} for each $x,y$ in $G$. Then we 
have the following possibilities:
\begin{enumerate}[i)]
\item
\begin{eqnarray*}
F_1(x)&=&\alpha \gamma m(x)+\beta\delta m(-x)+T(x)\\
F_2(x)&=&\alpha \delta m(x)+\beta \gamma m(-x)-T(x)\\
g(x)&=&\alpha m(x)+\beta m(-x)\\
h(x)&=&\gamma m(x)+\delta m(-x)
\end{eqnarray*}
\item
\begin{eqnarray*}
F_1(x)&=&\frac{1}{2}(a(x)+\alpha \beta+\gamma) m_0(x)+T(x)\\
F_2(x)&=&\frac{1}{2}(-a(x)+\alpha \beta-\gamma) m_0(x)-T(x)\\
g(x)&=&\alpha m_0(x)\\
h(x)&=&\left[\frac{1}{\alpha}a(x)+\beta\right] m_0(x)
\end{eqnarray*}
\item
\begin{eqnarray*}
F_1(x)&=&\frac{1}{2}(a(x)+\alpha \beta+\gamma) m_0(x)+T(x)\\
F_2(x)&=&\frac{1}{2}(a(x)+\alpha \beta-\gamma) m_0(x)-T(x)\\
g(x)&=&\left[\frac{1}{\alpha}a(x)+\beta\right] m_0(x)\\
h(x)&=&\alpha m_0(x)
\end{eqnarray*}
\item
\begin{eqnarray*}
F_1(x)&=&T(x)\\
F_2(x)&=&-T(x)\\
g(x)&=&0\\
h&=&\text{arbitrary function}
\end{eqnarray*}
\item
\begin{eqnarray*}
F_1(x)&=&T(x)\\
F_2(x)&=&-T(x)\\
g&=&\text{arbitrary function}\\
h(x)&=&0
\end{eqnarray*}
\end{enumerate}
for each $x$ in $G$, where $\alpha, \beta, \gamma, \delta$ are complex  numbers, ($\alpha\ne 0$ in (ii) and (iii)), $a\colon G\to\C$ is a nonzero additive function, $m, m_0:G\to \C$ are 
exponentials, with $m_0$ is even, $m\ne \widecheck{m}$, and $T:G\to\C$ is a  $2 G$-periodic function. Conversely, the functions given  with these properties satisfy the functional equation \eqref{dAlem1}. If, in addition, $G$ is a topological group and $g,h\ne 0$ are  continuous,  then $a,m,m_0$ are continuous, too. If $G$ is a locally compact group  and 
$g,h\ne 0$ are measurable, then $a,m,m_0,g,h$ are continuous. If $G$ is 2-divisible, then $T$ is constant and the given regularity properties hold  for $F_1, F_2$, too.
\end{theorem}

\begin{proof}
By Theorem \ref{dAlemevensol} and Theorem \ref{dAlemoddsol}, we know the 
possible forms of  $F=F_1+F_2$ and $H=F_1-F_2$, further
\begin{equation*}
F_1=\frac{1}{2}(F+H),\hskip.2cm F_2=\frac{1}{2}(F-H)\,.
\end{equation*} 
The point is that in the formulas given for $F$ and $H$ in Theorem  \ref{dAlemevensol} and Theorem \ref{dAlemoddsol} the function $g$ is the  same. We have to pair the cases given in Theorems \ref{dAlemevensol} and  \ref{dAlemoddsol} in such a way that $g$ has the same form given in the two  cases. In the following part of the proof we go through all possible pairings of 
the  cases in the two theorems above.
\vskip.3cm

In the first case we consider Case $(i)$ in Theorem \ref{dAlemevensol} and Case $(i)$ Theorem \ref{dAlemoddsol}, so that we have
\begin{equation*}
g(x)=\alpha m(x)+\beta m(-x)=\alpha' m'(x)+\beta' m'(-x)
\end{equation*}
for each $x$ in $G$, where $\alpha,\beta,\alpha',\beta'$ are constants, $m,m'$  are exponentials and $m\ne \widecheck{m}$, $m'\ne \widecheck{m}'$. By the  linear independence of different exponentials we have that in this case $m=m'$, 
or $\widecheck{m}=m'$. By symmetry, we may suppose that $m=m'$, hence  $\alpha=\alpha'$ and $\beta=\beta'$. It follows that in the formulas for $F$ and $H$ we have the same $m$, that is
\begin{eqnarray*}
F_1(x)&=&\alpha \gamma m(x)+\beta\delta m(-x)+T(x)\\
F_2(x)&=&\alpha \delta m(x)+\beta \gamma m(-x)-T(x)\\
g(x)&=&\alpha m(x)+\beta m(-x)\\
h(x)&=&\gamma m(x)+\delta m(-x)
\end{eqnarray*}
for each $x$ in $G$. Here $\alpha, \beta, \gamma, \delta$ are arbitrary  complex numbers, $m$ is an exponential and $T:G\to\C$ is a $2 G$-periodic function. This is Case $(i)$ in our statement.
\vskip.3cm

Now we pair Case $(i)$ in Theorem \ref{dAlemevensol} with Case $(ii)$ in Theorem  \ref{dAlemoddsol}. In this case we must have $\beta=0$ and  $m=\widecheck{m}=m_0$, by the linear independence of different  exponentials. However, $m\ne \widecheck{m}$ in Case $(i)$ of Theorem \ref{dAlemevensol}, hence this pairing is impossible.
\vskip.3cm

Pairing Case $(i)$ in Theorem \ref{dAlemevensol} with Case $(iii)$ in Theorem 
\ref{dAlemoddsol} gives $g=0$, hence $\alpha=\beta=0$, that is 
$F_1+F_2=0$, which gives immediately our Case $(iv)$ above. Finally, pairing of 
Case $(i)$ in Theorem \ref{dAlemevensol} with Case $(iv)$ in Theorem \ref{dAlemoddsol} yields Case $(i)$ 
in our statement with $\delta=\gamma=\frac{1}{2}\gamma'$, where $\gamma'$ denotes the constant from Theorem \ref{dAlemevensol} case $(i)$.
\vskip.3cm

Pairing Case $(ii)$ in Theorem \ref{dAlemevensol} with Case $(i)$ is impossible: by independence of exponentials we have  $m=\widecheck{m}=m_0$ but $m\ne \widecheck{m}$ in Case $(i)$ of Theorem \ref{dAlemevensol}.  Pairing Case $(ii)$ in Theorem \ref{dAlemevensol} with Case $(ii)$ in 
Theorem \ref{dAlemoddsol} gives Case $(ii)$ above. Pairing Case $(ii)$ in Theorem 
\ref{dAlemevensol} with Case $(iii)$
Theorem \ref{dAlemoddsol} gives Case $(iii)$ in our present theorem. Finally, 
pairing Case $(ii)$ in Theorem 
\ref{dAlemevensol} with Case $(iv)$
Theorem \ref{dAlemoddsol} gives Case $(iii)$ above with $\gamma=0$.
\vskip.3cm

Pairing Case $(iii)$ in Theorem \ref{dAlemevensol} with Case $(i)$, with Case $(iii)$, or 
with Case $(iv)$ in 
Theorem \ref{dAlemoddsol} results in Case $(iv)$ above, and pairing Case $(iii)$ in 
Theorem \ref{dAlemevensol} with Case $(ii)$ in 
Theorem \ref{dAlemoddsol} is impossible.
\vskip.3cm

Pairing Case $(iv)$ in Theorem \ref{dAlemevensol} with Case $(i)$, resp. with Case $(ii)$ in  
Theorem  \ref{dAlemoddsol} gives Case $(i)$, resp. Case $(ii)$ above. Finally, pairing 
Case $(iv)$ in Theorem \ref{dAlemevensol} with Case $(iii)$ in  
Theorem \ref{dAlemoddsol} gives Case $(iv)$ above, and pairing 
Case $(iv)$ in Theorem \ref{dAlemevensol} with Case $(iv)$ in  
Theorem \ref{dAlemoddsol} gives Case $(v)$ above.
\vskip.3cm

It is a simple calculation to check the in all cases listed above the given functions 
are solutions of the functional equation \eqref{dAlem1}. Finally, the regularity 
statements are consequences of the previous theorems.
\end{proof}

\section{Solution of Gajda-type equations}


In this section we apply our results to the functional equation
\begin{equation}\label{Fech}
\int  \textcolor{red}{[}f(x+y-t)+f(x-y+t)\textcolor{red}{]} d\mu(t)=f(x)k(y)\,,
\end{equation}
which is a special case of \eqref{sincos1} with the choice $f=f_1=f_2=g$ and $k=h$.  For the existence of the integral in \eqref{Fech} we can use different assumptions on the group $G$, the measure $\mu$ and the unknown functions $f,k$. Equation \eqref{Fech} was studied in \cite[Fechner]{MR2515239} on locally compact abelian groups with the assumption that $f,h$ are essentially bounded Haar measurable functions and $\mu$ is a regular bounded complex Borel measure. For the moment we assume that $G$ is an abelian group, further the measure $\mu$ on $G$ and the functions $f,k$ are such that the above integral exists for each $x,y$ in $G$. For instance, this is the case if $G$ is a topological abelian group, $\mu$ is a compactly supported Radon measure on $G$, and $f,k$ are continuous functions. 
\vskip.3cm

Our idea is to apply Theorem \ref{dAlemsol1}. Using the notation of Theorem  \ref{dAlemsol1} we have 
$$
F_1=f*\mu,\hskip.2cm F_2=\widecheck{f}*\mu,\hskip.2cm g=f,\hskip.2cm h=k\,.
$$ 
Obviously,  we may suppose that $f\ne 0$. In addition we suppose that $k\ne 0$, too. Then we 
have three possibilities given by Theorem \ref{dAlemsol1}.
\vskip.3cm

 In the first case
\begin{equation}\label{fform1}
f(x)=\gamma m(x)+\delta m(-x),\hskip.2cm k(x)=\alpha m(x)+\beta m(-x)\,,
\end{equation}
and, by the form of $F_1$ and $F_2$,
 we have
 \begin{equation*}
\alpha \gamma m(x)+\beta \delta m(-x)+T(x)=\gamma \widehat{\mu}(\widecheck{m}) m(x)+\delta \widehat{\mu}(m) m(-x)\,,
 \end{equation*} 
 further
\begin{equation*}
\alpha \delta m(x)+\beta \gamma m(-x)-T(x)=\gamma \widehat{\mu}(m) m(-x)+\delta \widehat{\mu}(\widecheck{m}) m(x)\,.
\end{equation*}
Here $\alpha,\beta,\gamma,\delta$ are complex numbers, where at least one of $\gamma$ and $\delta$ is nonzero, $m$ is a non-even exponential, and $T$ is $2 G$-periodic. Using the fact that $m$ and $\widecheck{m}$ are linearly independent, substitution into \eqref{Fech} gives the following necessary and sufficient condition for $f,k$ is a solution:
\begin{eqnarray*}
\gamma \widehat{\mu}(m)&=&\alpha \gamma\\
\gamma \widehat{\mu}(\widecheck{m})&=&\beta \gamma\\
\delta \widehat{\mu}(m)&=&\alpha \delta\\
\delta \widehat{\mu}(\widecheck{m})&=&\beta \delta\,.
\end{eqnarray*}
By the condition on $\gamma,\delta$, we infer $\alpha=\widehat{\mu}(m)$ and $\beta=\widehat{\mu}(\widecheck{m})$. In this case we have $T=0$, which is $2 G$-periodic and $f,k$ is a solution of \eqref{Fech}. We note that $f,k$ of the form obtained in this way is a solution also in the case, when $m$ is an even exponential, as it is easy to see.
\vskip.3cm

In the second case of Theorem \ref{dAlemsol1} we have
\begin{equation}\label{fform3}
f(x)=\alpha m_0(x),\hskip.2cm k(x)=\left[\frac{1}{\alpha}a(x)+\beta\right]m_0(x)\,,
\end{equation}
and, by the form of $F_1$ and $F_2$, we have
 \begin{equation*}
\left[\frac{1}{2}a(x)+\frac{1}{2}(\alpha \beta+\gamma)\right]m_0(x)+T(x)=\alpha m_0(x)\widehat{\mu}(m_0)\,,
 \end{equation*} 
 further
\begin{equation*}
\left[\frac{1}{2}a(x)+\frac{1}{2}(\alpha \beta-\gamma)\right]m_0(x)-T(x)=\alpha m_0(x)\widehat{\mu}(m_0)\,.
 \end{equation*} 
Here $\alpha,\beta,\gamma$ are complex numbers, where $\alpha$ is nonzero, $m_0$ is an even exponential, and $T$ is $2 G$-periodic.
Substitution into \eqref{Fech} gives that $a=0$ and $\beta=2\widehat{\mu}(m_0)$. In this case $T$ is $2 G$-periodic, hence we have a solution. 
However, this solution is included in the first case with $m_0=m=\widecheck{m}$, and $\alpha=\gamma+\delta$.

\vskip.3cm

Finally, in the third case of Theorem \ref{dAlemsol1} we have
\begin{equation}\label{fform2}
f(x)=\left[\frac{1}{\alpha}a(x)+\beta \right] m_0(x),\hskip.2cm k(x)=\alpha m_0(x)\,,
\end{equation}
and, by the form of $F_1$ and $F_2$,
 we conclude
 \begin{equation*}
\left[\frac{1}{2}a(x)+\frac{1}{2}(\alpha \beta+\gamma)\right]m_0(x)+T(x)=
 \end{equation*} 
 \begin{equation*}
\frac{1}{\alpha}a(x)m_0(x)\widehat{\mu}(m_0)+\beta m_0(x)\widehat{\mu}(m_0)-\frac{1}{\alpha} m_0(x) \int a(y)m_0(y) d\mu(y)\,,
 \end{equation*}
 further
\begin{equation*}
\left[-\frac{1}{2}a(x)+\frac{1}{2}(\alpha \beta-\gamma)\right]m_0(x)-T(x)=
 \end{equation*} 
 \begin{equation*}
-\frac{1}{\alpha}a(x)m_0(x)\widehat{\mu}(m_0)+\beta m_0(x)\widehat{\mu}(m_0)+\frac{1}{\alpha} m_0(x) \int a(y) m_0(y) d\mu(y)\,.
\end{equation*}
Here $\alpha,\beta,\gamma$ are complex numbers, where $\alpha$ is nonzero, $m_0$ is an even exponential, and $T$ is $2 G$-periodic. 
Substitution into \eqref{Fech} gives the following necessary and sufficient condition for $f,k$ is a solution: $\alpha=2\widehat{\mu}(m_0)$. 
In this case the two equations for $T$ hold true and $T$ is $2 G$-periodic. It follows that in this case we have a solution if and only if $\widehat{\mu}(m_0)$ is nonzero for some even exponential $m_0$.
\vskip.3cm

We can summarize our results on the equation \eqref{Fech} in the following 
result.

\begin{theorem}\label{Fechsol}
Let $G$ be an abelian group, $\mu$ a measure on $G$, and let $f,k:G\to \C$ be nonzero functions such that the integral in \eqref{Fech} exists for each $x,y$ in $G$, further \eqref{Fech} holds. Then we have the following 
possibilities
\begin{enumerate}[i)]
\item 
\begin{equation*}
f(x)= \gamma m(x)+ \delta m(-x),\hskip.2cm k(x)=\widehat{\mu}(m) m(x)+ \widehat{\mu}(\widecheck{m}) m(-x)
\end{equation*}
for each $x$ in $G$, where $m$ is an exponential, and 
$\gamma,\delta$ are complex numbers.
\item 
\begin{equation*}
f(x)=\left[\frac{1}{2\widehat{\mu}(m_0)}a(x)+\beta\right]m_0(x),\hskip.2cm k(x)=2 \widehat{\mu}(m_0) m_0(x)
\end{equation*}
for each $x$ in $G$, where $m_0$ is an even exponential with $\widehat{\mu}(m_0)\ne 0$, $a$ is an additive function, and $\beta$ is a complex number. 
\end{enumerate}
Conversely, the functions $f,k$ given above are solutions of 
\eqref{Fech}, whenever the given conditions are satisfied. If $G$ is a topological 
group, and $f$ or $k$ is continuous, then $a, m$ and $m_0$ are continuous, too. 
If $G$ is locally compact, and $f$ or $k$ is Haar measurable, then $f,k,a,m,m_0$ are  continuous. If $f$ or $k$ is essentially bounded and Haar measurable, then $a=0$,
$f,k,m,m_0$ are  continuous, moreover $m,m_0$ are characters of $G$. 
\end{theorem}

We note that the regularity statements follow from the above results, or directly 
from Lemma 5.5 (p.48) and Theorem 5.10 (p.51) in  \cite[Sz\'ekelyhidi]{MR1113488}.
\vskip.3cm
In a similar way we can obtain the more general solutions of equation \eqref{eq:WilsonSplotChange}, that is
\begin{equation}
\int  [f(x+y-t)+f(x-y+t) ] d\mu(t)=k(x)f(y),\quad x,y \in G\,,
	\label{eq:WilsonSplotChangeIntegral}
\end{equation}
Our preliminary assumption for the computations below is again that $G$ is an abelian group, further the measure $\mu$ on $G$ and the functions $f,k$ are such that the above integral exists for each $x,y$ in $G$. Interchanging $x$ and $y$ and using the notation

$$
F_1=f*\mu,\hskip.2cm F_2=\widecheck{f}*\mu,\hskip.2cm g=k,\hskip.2cm h=f
$$ 
we apply Theorem \ref{dAlemsol1} again to get the following result exactly in the same way as above.
\begin{theorem}\label{Fechsol2011}
Let $G$ be an abelian group, $\mu$ a measure on $G$, and let $f,k:G\to \C$ be nonzero functions such that the integral in \eqref{eq:WilsonSplotChangeIntegral} exists for each $x,y$ in $G$, further
\eqref{eq:WilsonSplotChangeIntegral} holds. Then we have the following 
possibilities
\begin{enumerate}[i)]
\item
\begin{equation*}
f(x)=\alpha\widehat{\mu}(\widecheck{m}) m(x)+ \alpha\widehat{\mu}(m) m(-x),\hskip.2cm k(x)=\widehat{\mu}(\widecheck{m})  m(x)+\widehat{\mu}(m) m(-x)\,,
\end{equation*}
for each $x$ in $G$, where $m$ is a non-even exponential, and $\alpha$ is a non-zero complex constant.
\item
\begin{equation*}
f(x)= 2 \alpha\widehat{\mu}(m_0) m_0(x),\hskip.2cm k(x)=2 \widehat{\mu}(m_0) m_0(x)
\end{equation*}
for each $x$ in $G$, where $m_0$ is an even exponential with $\widehat{\mu}(m_0)\ne 0$, and $\alpha$ is a non-zero complex constant. 
\end{enumerate}
Conversely, the functions $f,k$ given above are  solutions of 
\eqref{eq:WilsonSplotChangeIntegral}, whenever the given conditions are satisfied. If $G$ is a topological 
group, and $f$ or $k$ is continuous, then  $m$ and $m_0$ are continuous, too. 
If $G$ is locally compact, and $f$ or $k$ is Haar measurable, then $f,k,m,m_0$ are 
continuous. If $f$ or $k$ is essentially bounded and Haar measurable, then
$f,k,m,m_0$ are  continuous, moreover $m,m_0$ are characters of $G$. 
\end{theorem}
Observe that in case of equation \eqref{eq:WilsonSplotChangeIntegral} no additive function appears in the final form of the solution. The reason is that the function $g$ and $h$ being the solution of \eqref{dAlem1} described in Theorem \ref{dAlemsol1} cannot have simultaneously the same additive component.
\vskip.3cm
The above results cover also the previous research mentioned in the introduction. In the case of Gajda's 
equation \eqref{eq:Gajda} we have $f=k$ and the solution has the form
\begin{equation*}
f(x)=\widehat{\mu}(m) m(x)+\widehat{\mu}(\widecheck{m}) m(-x)
\end{equation*}
for each $x$ in $G$, where $m$ is an arbitrary exponential.
\vskip.3cm
In the case of d'Alembert's equation \eqref{eq:DalembertKlasyczny} the formula reduces to 
\begin{equation*}
f(x)=\frac{1}{2} \bigl(m(x)+ m(-x)\bigr)\,,
\end{equation*}
as in this case $\mu=\frac{1}{2}\delta_0$, hence
$$
\widehat{\mu}(m)=\frac{1}{2}\int  m(-y) d\delta_0(y)=\frac{1}{2} 
m(0)=\frac{1}{2}\,,
$$
and similarly $\widehat{\mu}(\widecheck{m})=\frac{1}{2}$.

\end{document}